\documentclass{amsart}
\usepackage{amsmath,amsthm,amssymb,amscd}
\usepackage[all,cmtip]{xy}
\usepackage{todonotes}

\usepackage[english]{babel}
\usepackage{mathrsfs}
\usepackage{amssymb}
\usepackage{diagxy}
\usepackage[margin=1.2in]{geometry}

\renewcommand{\lim}{\varprojlim}

\newcommand{\into}{\hookrightarrow}

\DeclareMathOperator{\Fc}{\mathcal{F}}

\newcommand{\Dir}{\mathsf{Dir}^{\textit{a}}}

\DeclareMathOperator{\A}{\mathsf{A}}

\DeclareMathOperator{\Xc}{\mathcal{X}}

\DeclareMathOperator{\Lex}{Lex}

\DeclareMathOperator{\Prok}{\mathsf{Pro}^{\textit{a}}_{\kappa}}

\newcommand{\Cc}{\mathcal{C}}

\newcommand{\Index}{\mathsf{Index}}

\newcommand{\Tatek}{\mathsf{Tate}_{\kappa}}
\newcommand{\elTatek}{\mathsf{Tate}^{\textit{el}}_{\kappa}}

\DeclareMathOperator{\coker}{coker}

\DeclareMathOperator{\ic}{ic}

\DeclareMathOperator{\Indk}{\mathsf{Ind}^{\textit{a}}_{\kappa}}
\DeclareMathOperator*{\colim}{\varinjlim}

\newcommand{\C}{\mathsf{C}}

\newcommand{\Gr}{Gr}

\DeclareMathOperator{\grp}{\times}

\DeclareMathOperator{\Pro}{\mathsf{Pro}^{\textit{a}}}

\DeclareMathOperator{\Ind}{\mathsf{Ind}^{\textit{a}}}
\DeclareMathOperator{\Fun}{Fun}

\newcommand{\ab}{\mathsf{Ab}}
\DeclareMathOperator{\op}{op}

\DeclareMathOperator{\QCoh}{QCoh}

\DeclareMathOperator{\Coh}{Coh}

\DeclareMathOperator{\F}{\mathcal{F}}

\DeclareMathOperator{\D}{\mathsf{D}}

\DeclareMathOperator{\Kk}{\mathcal{K}}

\DeclareMathOperator{\elTate}{\mathsf{Tate}^{\textit{el}}}

\newtheorem{definition}{Definition}[section]
\newtheorem{theorem}[definition]{Theorem}
\newtheorem{proposition}[definition]{Proposition}
\newtheorem{corollary}[definition]{Corollary}

\newtheorem{lemma}[definition]{Lemma}
\newtheorem{rmk}[definition]{Remark}

\newtheorem{example}[definition]{Example}
\let\oldexample\example
\renewcommand{\example}{\oldexample\normalfont}

\newtheorem*{acknowledgement}{Acknowledgement}

\begin{document}

\title[Relative Tate Objects]{Relative Tate Objects and Boundary Maps in the \\ $K$-Theory of Coherent Sheaves}
\author{Oliver Braunling,\quad Michael Groechenig,\quad Jesse Wolfson}

\address{Department of Mathematics, Universit\"{a}t Freiburg}
\email{oliver.braeunling@math.uni-freiburg.de}
\address{Department of Mathematics, Imperial College London}
\email{m.groechenig@imperial.ac.uk}
\address{Department of Mathematics, University of Chicago}
\email{wolfson@math.uchicago.edu}

\begin{abstract}
We investigate the properties of relative analogues of admissible Ind, Pro, and elementary Tate objects for pairs of exact categories, and give criteria for those categories to be abelian. A relative index map is introduced, and as an application we deduce a description for boundary morphisms in the $K$-theory of coherent sheaves on Noetherian schemes.
\end{abstract}

\thanks{O.B. was supported by GK1821 "Cohomological Methods in Geometry". M.G.\ was supported by EPSRC Grant No.\ EP/G06170X/1. J.W.\ was partially supported by an NSF Graduate Research Fellowship under Grant No.\ DGE-0824162, by an NSF Research Training Group in the Mathematical Sciences under Grant No.\ DMS-0636646, and by an NSF Post-doctoral Research Fellowship under Grant No.\ DMS-1400349. J.W. was a guest of K. Saito at IPMU while this paper was being completed. Our research was supported in part by NSF Grant No.\ DMS-1303100 and EPSRC Mathematics Platform grant EP/I019111/1.}

\maketitle

\tableofcontents

\section{Introduction}

If one knows the category $\Coh(X)$ of coherent sheaves of a scheme,
the quasi-coherent sheaves are just the category of Ind objects $\QCoh%
(X)\simeq\mathsf{Ind}^{a}\left(  \Coh(X)\right)  $, i.e. they arise from
an entirely formal categorical process. What is the geometric role of the
Pro object analogue?\medskip

Two examples: (1) For $j:U\hookrightarrow X$ open, Deligne \cite{deligne}
defines an extension-by-zero functor $j_{!}$, a type of left adjoint for the
pull-back $j^{\ast}$, or (2) for $i:Z\hookrightarrow X$ a closed immersion the adic
completion naturally outputs a Pro-coherent sheaf:
\[
j_{!}:\Coh(U)\longrightarrow\mathsf{Pro\,}\Coh(X)\qquad
\qquad\mathbf{C}_{Z}:\Coh(X)\longrightarrow\mathsf{Pro\,}%
\Coh(Z)\text{.}%
\]
Both functors \textquotedblleft need\textquotedblright\ Pro objects and cannot
be defined inside coherent sheaves alone, e.g. for $j_{!}$ this is forced by
the adjunction property. Although both functors are very natural, Pro-coherent
sheaves are used far less often than their Ind-counterpart in practice. Two
natural questions arise:

\begin{enumerate}
\item Is there a natural framework allowing one to view both Ind- and Pro-coherent
sheaves as objects in \textit{one} category?

\item How do the notions of Ind- and Pro-coherent sheaves generalise for
sheaves with support?
\end{enumerate}

This article proposes an answer to these questions and studies the effect of
these functors on algebraic $K$-theory. We summarize our answers:\medskip

(1) For this there would trivially be a boring answer by just taking an
extremely large category. However, we shall argue that the category of Tate
objects $\elTate(\Coh(X))$, originally introduced by A.
Beilinson \cite{MR923134} and K. Kato \cite{MR1804933} for different reasons,
is an interesting candidate. This is an exact category whose $K$-theory has
close ties to that of $\Coh(X)$, and its objects are precisely
extensions of quasi-coherent sheaves by Pro-coherent sheaves. So, in a way it provides
the \textsl{minimal} solution to our question. However, while
$\mathbf{C}_{Z}$ takes values in this category, this is \textit{not} the case
for Deligne's $j_{!}$. See Equation \ref{eqn:deligne}, and the paragraph after Theorem \ref{thm:introeasyreading} for a precise explanation of how this functor is related to our work.

(2) There are several ways to weave support constraints into these
categories, e.g. $\mathsf{Ind}^{a}\left(  \Coh_{Z}X\right)  $ are
Ind objects from coherent sheaves supported in $Z$, while we shall also
introduce a category $\mathsf{Ind}^{a}(\Coh(X),\Coh_{Z}X)$,
which contains Ind objects built of arbitrary coherent sheaves, but so that
the Ind-system has relative quotients with supports in $Z$. There are a number
of further variations of the theme of support and we investigate the relations
between these categories. Ultimately this requires a \textit{relative} Tate
category $\elTate(\mathsf{D},\mathsf{C})$, which gives this article
its name.

Once these categories are properly constructed, we use them to address a
question in algebraic $K$-theory. Namely, any open-closed complement
$U\hookrightarrow X\hookleftarrow Z$ gives rise to a localization sequence%
\[
\cdots\longrightarrow G_{i}(Z)\longrightarrow G_{i}(X)\longrightarrow
G_{i}(U)\overset{\partial}{\longrightarrow}G_{i-1}(Z)\longrightarrow\cdots
\]
in the $K$-theory of coherent sheaves, i.e. $G_{i}(X):=\pi_{i}(K_{\Coh(X)})$.

\begin{theorem}\label{thm:introeasyreading}
Suppose $X$ is a Noetherian scheme, $U\overset{j}{\hookrightarrow}X$ an open subscheme and $Z$ its reduced closed complement. Then there is a canonical homotopy commutative
triangle%
\[
\xymatrix{
\Omega K_{\Coh(U)} \ar[d]_{\textbf{T}_Z} \ar[rr]^\partial
&& K_{\Coh(Z)} \\
\Omega K_{\elTate(\Coh(X),\Coh_{Z}%
(X))}, \ar[urr]^{\sim}_{\mathsf{i}}
}
\]
where

\begin{enumerate}
\item the map $\mathbf{T}_{Z}$ is the \textquotedblleft open
complement\textquotedblright\ to the adic completion functor $\mathbf{C}_{Z}$;
it will be defined below in Corollary \ref{cor:T_Z} (or see Equation \ref{delignereleq} below),

and \item $\mathsf{i}$ is an equivalence of $K$-theory spaces, given as a concrete
zig-zag of simplicial maps in \S \ref{relative}.
\end{enumerate}
\end{theorem}

The functor $\mathbf{T}_Z$ is in general only exact if the inclusion $j\colon U \rightarrow X$ is affine. However, exactness fails only up to a part irrelevant to $K$-theory, which allows us to state the theorem without imposing this assumption. See Corollary \ref{cor:G-theory} for the
full formulation

It is the functor $\mathbf{T}_Z$ which is related to Deligne's $j_!$. As we will see in Equation \ref{eqn:deligne}, there is a short exact sequence of (not necessarily admissible) Ind Pro objects
\begin{equation}\label{delignereleq}
0 \rightarrow j_!\F \rightarrow j_*\F \rightarrow \mathbf{T}_Z(\F) \rightarrow 0.
\end{equation}

\medskip

Let us explain the simplicial map of the theorem a little more. We keep the assumptions as said. Denote by  $\Gr^{\le}_{\bullet,\bullet}(X,Z)$ the bi-simplicial space given by the nerves of the groupoids of collections of coherent sheaves $(\F_{ij})_{i \leq n,\text{ }j \leq m}$ with inclusions $\F_{ij} \subset \F_{kl}$ for $i \leq k$ and $ j \leq l$. Moreover we assume that $j^*\F_{ij} = j^*\F_{kj}$ for $i \leq k$. Consider the span of maps
$S_{\bullet}\Coh(U)^{\times} \leftarrow \Gr^{\le}_{\bullet,\bullet}(X,Z) \rightarrow S_{\bullet}S_{\bullet}\Coh_Z(X),$
where the left-pointing arrow sends the above diagram to the simplicial diagram obtained by restriction along $j^*$. The rightward arrow maps the above to $(\F_{ij}/\F_{0j}) \in \Coh_Z(X)$.

The face and degeneracy maps of $\Gr^{\le}_{\bullet,\bullet}(X,Z)$ are defined, such that for each fixed index $i$, the map $(\F_{kj})_{k,j} \mapsto (\F_{ij})_j$ defines a map $\Gr^{\le}_{\bullet,\bullet}(X,Z) \rightarrow S_{\bullet}(\Coh(X))^{\grp}$. We now obtain the following reformulation of Theorem \ref{thm:introeasyreading}.

\begin{theorem}\label{thm:main_intro}
Taking geometric realisations of the aforementioned map, and applying the double loop space functor $\Omega^2$, we obtain the following homotopy commutative diagram
\[
\xymatrix{
 & \Omega^2|\Gr^{\le}_{\bullet,\bullet}(X,Z)| \ar[ld]_{\simeq} \ar[rd] & \\
\Omega K_{\Coh(U)} \ar[rr]^{\partial} & & K_{Z},
}
\]
where $\partial\colon K_{\Coh(U)} \rightarrow BK_{Z}$ is the boundary map of Theorem \ref{thm:introeasyreading}, resp. in the $G$-theory localisation sequence.
\end{theorem}

In order to establish our main result we continue the investigations of our article \cite{TateObjectsExactCats}, which was devoted to a detailed analysis of the exact categories of elementary Tate objects. For an exact category $\D$, together with an extension-closed full subcategory $\C \subset \D$, we define and study exact categories of \emph{relative} Tate objects $\elTate(\D,\C) \subset \elTate(\D)$; as well as their cousins $\Ind(\D,\C)$ and $\Pro(\D,\C)$ of relative admissible Ind and Pro objects. We also give necessary and sufficient criteria for categories of relative admissible Ind objects to be abelian, namely that $\C$ and $\D$ are abelian and satisfy a relative analogue of the Noetherian condition; dually, relative admissible Pro objects are abelian if and only if $\C$ and $\D$ satisfy a relative Artinian condition.  We refer the reader to Definitions \ref{defi:relative_tate} and \ref{defi:noetherian} for a precise explanation of those terms.

An example of particular interest to us is given by $\Coh_Z(X) \subset \Coh(X)$, where $X$ is a Noetherian scheme, and $Z$ a closed subscheme. We denote by $\Coh_Z(X)$ the full subcategory of coherent sheaves on $X$ with set-theoretic support at $Z$; and by $\QCoh(X,X\setminus Z)$ the category of quasi-coherent sheaves on $X$, whose restriction to $X \setminus Z$ is coherent. The following statement is part of Proposition \ref{prop:example} in the main body of the text.
\begin{example}
We have an exact equivalence $\QCoh(X,X\setminus Z) \cong \Ind(\Coh(X),\Coh_Z(X))$.
\end{example}
The study of elementary Tate objects in exact categories was pioneered by Beilinson \cite{MR923134}, who introduced this notion in order to study vanishing cycles. This direction was then further pursued by Previdi \cite{MR2872533}, and the authors in \cite{TateObjectsExactCats}. We view the present article as a natural continuation of these investigations.

\section{Recollection on Exact Categories}

For the remainder of this section we fix exact categories $\C$, $\D$. The basic definitions and properties can be found in B\"uhler's survey \cite{MR2606234}. We denote by $\Lex \C$ the abelian category of left exact presheaves, that is functors $\C^{\op} \xrightarrow{F} \ab$, taking values in the category of abelian groups, such that a short exact sequence $X \hookrightarrow Y \twoheadrightarrow Z$ is sent to an exact sequence
$$0 \rightarrow F(Z) \rightarrow F(Y) \rightarrow F(X).$$
We recall the following definition and lemma from \cite{TateObjectsExactCats}.

\begin{definition}
Fix an infinite cardinal $\kappa$, and consider a filtered poset $I$ with $|I| \leq \kappa$. An \emph{admissible Ind diagram} is a functor $X\colon I \rightarrow \C$, such that for every $i \leq j$ we have that $X_i \hookrightarrow X_j$ is an admissible monomorphism. The full and extension-closed subcategory of $\Lex \C$, consisting of objects $X$ which can be represented by $\colim_{i \in I} X_i$, over an admissible Ind diagram with $|I| \leq \kappa$, will be denoted by $\Indk(\C)$. It will be referred to as the exact category of \emph{admissible Ind objects}.
\end{definition}

It is not obvious to see that $\Indk(\C) \subset \Lex \C$ is extension-closed (see \cite[Theorem 3.7]{TateObjectsExactCats}). As a corollary one obtains a canonical structure of an exact category on $\Indk(\C)$, inherited from the abelian category $\Lex \C$. We will often refer to the following result, which is Lemma 3.11 in \cite{TateObjectsExactCats}.

\begin{lemma}\label{lemma:ind_mono}
If $(X_i)_{i \in I}$ is an admissible Ind diagram in an exact category $\C$, then for every $i \in I$ the induced $X_i \hookrightarrow X$ is an admissible monomorphism in $\Ind(\C)$.
\end{lemma}

The definition of admissible Pro, and elementary Tate objects will be evidently modelled on this one. In fact, categorical duality allows one to define admissible Pro objects at no extra cost. The concept of elementary Tate objects combines Ind and Pro directions in a non-trivial manner.

\begin{definition}\label{defi:elTate}
We define $\Prok(\C) = (\Indk(\C^{\op}))^{\op}$, and refer to it as the exact category of admissible Pro objects. The full subcategory of $\Indk\Prok(\C)$, consisting of objects $V$ which sit in an exact sequence
$$0 \rightarrow L \rightarrow V \rightarrow V/L \rightarrow 0,$$
with $L \in \Prok(\C)$ and $V/L \in \Indk(\C)$, will be referred to by $\elTate(\C)$, the category of \emph{elementary Tate objects}. Any such admissible subobject $L$ of $V$ is called a \emph{lattice} in $V$. We denote the set of lattices by $\Gr(V)$.
\end{definition}

There are several equivalent ways to introduce elementary Tate objects. In \cite[Definition 5.2]{TateObjectsExactCats}, elementary Tate object were defined by means of the property described in the remark below. The fact that the two viewpoints are equivalent, is implied by \cite[Theorem 5.5]{TateObjectsExactCats}.

\begin{rmk}\label{rmk:elTate}
Every elementary Tate object can be represented as a directed colimit $\colim_{i \in I}X_i \in \Ind\Pro(\C)$ over a directed poset $I$, with $X_i \in \Pro(\C)$, such that for $i \leq j$ we have $X_j/X_i \in \C$.
\end{rmk}

We need to recall the notions of left and right $s$-filtering subcategories, which were introduced in \cite[Definition 1.3 \& 1.5]{MR2079996}. The definition given below differs from the one of \emph{loc. cit.} The fact that the two definitions are equivalent is due to B\"uhler, a proof is given in \cite[Appendix A]{TateObjectsExactCats}.

\begin{definition}\label{defi:filt}
Let $\C \hookrightarrow \D$ be a fully faithful, exact inclusion of exact categories.
\begin{itemize}
\item[(a)] We say that $\C$ is \emph{left filtering} in $\D$, if every morphism $X \rightarrow Y$ in $\D$, with $X \in \C$ factors through an admissible monomorphism $X' \hookrightarrow Y$ with $X' \in \C$. We say that $\C \subset \D$ is right filtering, if $\C^{\op} \subset \D^{\op}$ is left filtering.
\item[(b)] The inclusion $\C \subset \D$ is \emph{left special}, if for every admissible epimorphism $G \twoheadrightarrow Z$ in $\D$, with $Z \in \C$, there exists a commutative diagram with short exact rows
\[
\xymatrix{
0 \ar[r] & X \ar[r] \ar[d] & Y \ar[r] \ar[d] & Z \ar[r] \ar@{=}[d] & 0 \\
0 \ar[r] & F \ar[r] & G \ar[r] & Z \ar[r] & 0,
}
\]
with the top row being a short exact sequence in $\C$. We say that $\C$ is right special in $\D$ if $\C^{\op} \subset \D^{\op}$ is left special.
\item[(c)] If $\C \subset \D$ is simultaneously left special and left filtering, then we refer to it as left $s$-filtering. Dually, if $\C^{\op} \subset \D^{\op}$ is left $s$-filtering, then we say that $\C \subset \D$ is right $s$-filtering.
\end{itemize}
\end{definition}

Left or right $s$-filtering inclusions enable us to define a quotient exact category $\D/\C$ (see \cite[Proposition 1.16]{MR2079996}). Moreover, they play an important role in the abstract study of elementary Tate objects. Before expanding on this, we have to record an elementary property of $s$-filtering embeddings.

\begin{lemma}\label{lemma:reflected}
Let $\C \hookrightarrow \D$ be left, respectively right $s$-filtering, then the inclusion reflects admissible monomorphisms and admissible epimorphisms.
\end{lemma}

\begin{proof}
By categorical duality we may assume that $\C \subset \D$ is left $s$-filtering. If $X \hookrightarrow Y \twoheadrightarrow Z$ is a short exact sequence in $\D$, and $Y \in \C$, then one obtains that $X, Y \in \C$. Indeed, this is demanded by Schlichting's definition \cite[Definition 1.3 \& 1.5]{MR2079996}, or follows from \cite[Appendix A]{TateObjectsExactCats}, for the definition we stated above. In particular, if $X \hookrightarrow Y$ is an admissible monomorphism in $\D$ with $X$ and $Y$ in $\C$, it fits in a short exact sequence $X \hookrightarrow Y \twoheadrightarrow Z$ with $Z \in \C$ as well. But, the proof of Lemma 2.14 in \cite{TateObjectsExactCats} shows that this short exact sequence is also a short exact sequence in $\C$. Hence, we obtain that $X \hookrightarrow Y$ is also an admissible monomorphism in $\C$. Similarly, one deals with the case of admissible epimorphisms.
\end{proof}

It was shown in \cite[Proposition 5.6]{TateObjectsExactCats} that Pro objects are left
filtering in $\mathsf{Tate}^{el}(\mathcal{C})$ and Ind objects right
filtering. The following result strengthens these two facts. It develops the
idea of the proof of \cite[Prop. 5.8]{TateObjectsExactCats}, and is a statement of
independent interest:

\begin{proposition}
\label{Prop_LatticeLeftFiltAndRightFilt}Let $\mathcal{C}$ be an exact category.

\begin{enumerate}
\item Every morphism $Y\overset{a}{\longrightarrow}X$ in $\mathsf{Tate}%
^{el}(\mathcal{C})$ with $Y\in\mathsf{Pro}^{a}(\mathcal{C})$ can be factored
as $Y\overset{\tilde{a}}{\rightarrow}L\hookrightarrow X$ with $L$ a lattice in
$X$.

\item Every morphism $X\overset{a}{\longrightarrow}Y$ in $\mathsf{Tate}%
^{el}(\mathcal{C})$ with $Y\in\mathsf{Ind}^{a}(\mathcal{C})$ can be factored
as $X\twoheadrightarrow X/L\overset{\tilde{a}}{\rightarrow}Y$ with $L$ a
lattice in $X$.
\end{enumerate}
\end{proposition}

\begin{proof}
(1) By Remark \ref{rmk:elTate} the elementary Tate object $X$ can be represented as a formal colimit $\colim_{i \in I} X_i$, where $I$ is a filtered poset, and $I \rightarrow \Prok(\C)$ a diagram which satisfies the two conditions
\begin{itemize}
\item[-] For every pair $i \leq j$ in $I$ the induced morphism $X_i \rightarrow X_j$ is an admissible monomorphism in $\Prok(\C)$,
\item[-] such that the quotient $X_j/X_i$ lies in the full subcategory $\C$.
\end{itemize}
By virtue of the definition of morphisms in categories of Ind objects we see that $Y \rightarrow X$ factors through an $X_i \rightarrow X$. It remains to show that the latter map is an inclusion of a lattice. Lemma 3.11 of \cite{TateObjectsExactCats} implies that $X_i \hookrightarrow X$ is an admissible monomorphism. The quotient object $X/X_i$ is represented by the Ind system $\colim_{j \geq i} X_j/X_i$, and is hence an admissible Ind object in $\C$.
(2)
Since Ind objects are right filtering in $\mathsf{Tate}^{el}\left(
\mathcal{C}\right)  $ \cite[Proposition 5.8]{TateObjectsExactCats}, it suffices to
deal with the case $X\twoheadrightarrow Y$. We pick some lattice $L$ in $X$,
and by the right filtering of Ind objects again, we get an object
$C\in\mathsf{Ind}^{a}\left(  \mathcal{C}\right)  $ so that the diagram%
\[
\bfig\node l(0,500)[L]
\node c(500,500)[C]
\node x(0,0)[X]
\node y(500,0)[Y]
\node xl(0,-500)[X/L]
\arrow/^{ (}->/[l`x;]
\arrow/>/[c`y;]
\arrow/>>/[x`y;]
\arrow/>>/[l`c;]
\arrow/>>/[x`xl;]
\efig
\]
commutes. Being a quotient of a Pro object at the same time, we must have
$C\in\mathcal{C}$, since $\Pro(\C) \subset \Ind \Pro(\C)$ is left $s$-filtering by \cite[Proposition 3.10]{TateObjectsExactCats}. Let $W:=\ker(L\twoheadrightarrow C)$ and complete the
diagram to%
\[
\bfig\node l(0,500)[L]
\node c(500,500)[C]
\node x(0,0)[X]
\node y(500,0)[Y]
\node xl(0,-500)[X/L]
\node w(-500,500)[W]
\node xw(500,-500)[X/W]
\arrow/^{ (}->/[l`x;]
\arrow/^{ (}->/[w`l;]
\arrow/^{ (}->/[w`x;]
\arrow/>/[c`y;]
\arrow/>>/[x`y;]
\arrow/>>/[l`c;]
\arrow/.>/[xw`xl;]
\arrow/>>/[x`xl;]
\arrow/>>/[x`xw;]
\arrow/.>/[xw`y;]
\efig
\]
The existence of the arrows originating from $X/W$ follows from the universal
property of cokernels. In the idempotent completion
$X/W\rightarrow X/L$ must be an admissible epic by \cite[Prop. 7.6]{MR2606234}
and the Snake Lemma applied to%
\[
\bfig\node w(0,0)[W]
\node x(500,0)[X]
\node xw(1000,0)[X/W]
\node l(0,-500)[L]
\node x2(500,-500)[X]
\node xl(1000,-500)[X/L]
\arrow/^{ (}->/[w`x;]
\arrow/>>/[x`xw;]
\arrow/^{ (}->/[l`x2;]
\arrow/>>/[x2`xl;]
\arrow/^{ (}->/[w`l;]
\arrow/>>/[xw`xl;]
\arrow/=/[x`x2;]
\efig
\]
provides us with an isomorphism $\ker(X/W\twoheadrightarrow X/L)\cong L/W\cong
C$. We conclude that $C\hookrightarrow X/W\twoheadrightarrow X/L$
is exact already in elementary Tate objects, since inclusion in the idempotent completion of an exact category reflects exactness (see \cite[Proposition 6.13]{MR2606234}). Since $X/L$ is an Ind object and
$C\in\mathcal{C}$, it follows that $X/W$ is an Ind object. Moreover, $W$ is a
subobject of $L$, so $W$ is a Pro object and therefore a lattice. The
factorisation $X\twoheadrightarrow X/W\rightarrow Y$ proves the claim.
\end{proof}

Central to the theory of elementary Tate objects is the main result of \cite[Theorem 6.7]{TateObjectsExactCats}. For convenience of the reader we recall the main idea behind the argument.

\begin{theorem}
Suppose $\mathcal{C}$ is an idempotent complete exact category and
$V\in\mathsf{Tate}^{el}(\mathcal{C})$. Then the Sato Grassmannian $Gr(V)$ is a
directed and co-directed poset.
\end{theorem}

\begin{proof}[Sketch]
Suppose $L_{1},L_{2}\hookrightarrow V$ are lattices. Then $L_{1}\oplus
L_{2}\rightarrow V$ is a morphism from a Pro object to an elementary Tate
object. Thus, by Prop. \ref{Prop_LatticeLeftFiltAndRightFilt} there exists a
factorisation%
\[
L_{1}\oplus L_{2}\rightarrow L^{\prime}\hookrightarrow V
\]
with $L^{\prime}$ a lattice in $V$. Invoking the non-trivial result
\cite[Lemma 6.9]{TateObjectsExactCats} implies that the morphisms
$L_{i}\rightarrow L^{\prime}$ must be admissible monics.
\end{proof}

\section{Relative Tate Objects}
In this section we introduce relative versions of Ind, Pro, and Tate objects. This will allow us to give an index-theoretic description of boundary maps in algebraic $K$-theory. We begin by stating two lemmas on ordinary admissible Ind objects.

\begin{lemma}\label{lemma:IndCD}
    Let $\C$ and $\D$ be exact categories, and let $\C\hookrightarrow\D$ be an exact, fully faithful embedding. Then $\C\simeq\Ind(\C)\cap\D\subset\Ind(\D)$.
\end{lemma}
\begin{proof}
    Let $X\in\Ind(\C)$ be the colimit of an admissible Ind diagram $X\colon I\to \C$. Let $Y\in\D$ such that there exists $Y\to^\cong X$ in $\Ind(\D)$. By the definition of morphisms in $\Indk(\D)$, there exists $i \in I$, such that we have a factorisation
    \begin{equation*}
        \xymatrix{
            Y \ar[r]^{\cong} \ar[d] & X \\
            X_i. \ar@{^{(}->}[ru]
        }
    \end{equation*}
    The diagonal arrow is an admissible monic in $\Ind(\C)$ by construction (\cite[Lemma 3.11]{TateObjectsExactCats}); and the commutativity of the above diagram implies that it is also an (not necessarily admissible) epic. It is therefore an isomorphism.
\end{proof}

The next lemma is a slight generalization of \cite[Proposition 5.9(1)]{TateObjectsExactCats}; to shake things up, we give a different proof.
\begin{lemma}\label{lemma:IndCinIndD}
    Let $\D$ be an exact category, and let $\C\subset\D$ be a right (or left) s-filtering subcategory. Then for any short exact sequence in $\Ind(\D)$
    \begin{equation*}
        X\hookrightarrow Y\twoheadrightarrow Z
    \end{equation*}
    we have that $Y\in\Ind(\C)$ if and only if $X$ and $Z$ are.
\end{lemma}

\begin{proof}
We will show that the category $\Ind(\C)$ is equivalent to the fibre product
\[
\xymatrix{
\Ind(\C) \ar[r] \ar[d] & 0 \ar[d] \\
\Ind(\D) \ar[r]^F & \Ind(\D/\C),
}
\]
that is, it is the full subcategory of $X \in \Ind(\D)$, which are mapped to a zero object by $F$, i.e. $F(X) \cong 0 \in \Ind(\D/\C)$. Taking this for granted, we observe that for an exact sequence $X \hookrightarrow Y \twoheadrightarrow Z$, we have $F(Y) \cong 0$, if and only if $F(X)\cong 0$ and $F(Z)\cong 0$. Because $F$ is an exact functor, this implies that $Y \in \Ind(\C)$ if and only if $X \in \Ind(\C)$ and $Z \in \Ind(\C)$.

By \cite[Proposition 3.16]{TateObjectsExactCats}, we have a fully faithful functor $\Ind(\C) \hookrightarrow \Ind(\D)$, therefore it suffices to show that its essential image is precisely the kernel of the functor $F \colon \Ind(\D) \rightarrow \Ind(\D/\C)$. If $(X_i)_{i \in I}$ is an admissible Ind diagram in $\C$, then we have $F(X_i) \cong 0$ for every $i \in I$. In particular, $F(X_i)_{i \in I} \cong 0 \in \Ind(\D/\C)$.

Vice versa, if $(X_i)_{i \in I}$ is an admissible Ind diagram in $\D$, which is mapped to $0 \in \Ind(\D/\C)$, then the fact that the induced maps $F(X_i) \hookrightarrow 0$ are admissible monomorphisms (see Lemma \ref{lemma:ind_mono}) implies that $F(X_i) = 0$ for every $i \in I$. In particular, we see that every $X_i \in \C$. Admissible monomorphisms are reflected by right/left $s$-filtering inclusions (Lemma \ref{lemma:reflected}), which implies that $(X_i)_{i \in I} \in \Ind(\C)$.
\end{proof}

Having dealt with the technicalities above, we are ready to define relative Ind, Pro and Tate objects.
\begin{definition}\label{defi:relative_tate}
    Let $\D$ be an exact category, and let $\C \subset \D$ be an extension-closed subcategory. Let $\kappa$ be an infinite cardinal.
    \begin{itemize}
        \item[(a)] Define the category of \emph{relative admissible Ind objects} $\Indk(\D,\C)$ to be the full subcategory of $\Indk(\D)$ consisting of objects that admit a presentation by an admissible Ind diagram $X\colon I\to \D$ (cf. \cite[Def. 3.2]{TateObjectsExactCats}) such that for all $i<j$ in $I$, we have $X_j/X_i\in\C$.
        \item[(b)] Define the category of \emph{relative admissible Pro objects} $\Prok(\D,\C)$ by
            \begin{equation*}
                \Prok(\D,\C):=(\Indk(\D^{\op},\C^{\op}))^{\op}.
            \end{equation*}
        \item[(c)] Define the category of \emph{relative elementary Tate objects} $\elTatek(\D,\C)$ to be the category $\Indk(\Prok(\D),\C)$.
        \item[(d)] For $\C$ and $\D$ idempotent complete, define the category of \emph{relative Tate objects} $\Tatek(\D,\C)$ to be the idempotent completion $\elTatek(\D,\C)^{\ic}$.
    \end{itemize}
\end{definition}

\begin{rmk}
In the language of Definition \ref{defi:relative_tate}, the category of elementary Tate objects in $\C$ can be written as
    \begin{equation*}
        \elTatek(\C) = \Indk(\Prok(\C),\C).
    \end{equation*}
\end{rmk}

We begin with a formal observation, which characterises relative admissible Ind objects in categorical terms. Often it can be used to shorten proofs, provided one accepts stronger assumptions for the embedding $\C \hookrightarrow \D$.

\begin{lemma}
Assume that $\C \hookrightarrow \D$ is left or right $s$-filtering. Then, we have $\Indk(\D,\C) \cong \Indk(\D) \times_{\Indk(\D/\C)} \D/\C \hookrightarrow \Indk(\D)$ as a full subcategory of $\Indk(\D)$.
\end{lemma}

\begin{proof}
Let $(X_i)_{i \in I}$ be an Ind diagram, representing an object $X$ of $\Indk(\D)$. We denote the exact functor $\Indk(\D) \rightarrow \Indk(\D/\C)$ by $F$. If $(X_i)_{i \in I}$ is relatively admissible, then $(F(X_i))_{i \in I}$ is a constant diagram, since the transition maps $X_i \rightarrow X_j$ are mapped to equivalences in $\D/\C$. Conversely, if $\colim_I F(X_i)\cong Y$ for some $Y\in \D/\C$, then the isomorphism $Y\to \colim_I F(X_i)$ factors through the inclusion $F(X_i)\into \colim_I F(X_i)$ for some $i$. By Lemma \ref{lemma:ind_mono}, this inclusion is therefore an epic admissible monic, and thus an isomorphism.  Because $\D/\C \hookrightarrow \Indk(\D/\C)$ is left $s$-filtering by \cite[Proposition 3.10]{TateObjectsExactCats}, we can also conclude that $F(X_j)\in\D/\C$ for all $j\ge i$ in $I$, and that $F(X_j)_{j\ge i}$ is isomorphic in $\D/\C$ to a constant diagram.  We conclude that, for all $j\ge i$, $X_j/X_i\in \C$ and thus that $(X_j)_{j\ge i}$ is an admissible relative Ind-diagram.
\end{proof}

If the reader is willing to work instead with the assumption that $\C \hookrightarrow \D$ is left or right $s$-filtering, the next lemma is a direct consequence of the result proven above.

\begin{lemma}\label{lemma:relative_exact}
    Let $\C \subset \D$ be an extension-closed full subcategory. For any cardinal $\kappa$, $\Indk(\D,\C)$ is closed under extensions in $\Indk(\D)$. Similarly, $\Prok(\D,\C)$ is closed under extensions in $\Prok(\D)$ and $\elTatek(\D,\C)$ is closed under extensions in $\Indk(\Prok(\D))$.
\end{lemma}

Consequently we obtain an exact structure on the categories $\Indk(\D,\C)$, and similarly for relative admissible Pro objects, and relative elementary Tate objects.

\begin{corollary}
    The categories $\Indk(\D,\C)$, $\Prok(\D,\C)$ and $\elTatek(\D,\C)$ are exact categories.
\end{corollary}
\begin{proof}[Proof of Lemma \ref{lemma:relative_exact}]
    The statements about relative Pro and relative elementary Tate objects are special cases of the statement about relative Ind objects. In all cases, the lemma follows from the straightening construction for exact sequences \cite[Prop. 3.12]{TateObjectsExactCats} and the fact that $\C$ is closed under extensions in $\D$.

    In more detail, consider an exact sequence in $\Indk(\D)$
    \begin{equation*}
        0\to\widehat{X}\to\widehat{Y}\to\widehat{Z}\to 0
    \end{equation*}
    with $\widehat{X}$ and $\widehat{Z}$ in $\Indk(\D,\C)$. Let
    \begin{align*}
        X&\colon J\to \D\text{, and}\\
        Z&\colon I\to \D
    \end{align*}
    be admissible relative Ind diagrams. The straightening construction for exact sequences (Proposition 3.12 of \cite{TateObjectsExactCats}) shows that there exists a directed partially ordered set $K$ with final maps $K\to J$ and $K\to I$ such that the exact sequence above is isomorphic to the colimit of an admissible Ind diagram of exact sequences
    \begin{equation*}
        \xymatrix{
            K \ar[rd]^{\Rightarrow} \ar@{=}[r] & K \ar[d] \ar@{=}[r] & K \ar[ld]_{\Rightarrow}. \\
            & \C &
            }
    \end{equation*}
    For any map $i\le j$ in $K$, the $3\times 3$-Lemma \cite[Cor. 3.6]{MR2606234} shows that we have a diagram with exact rows and columns
    \begin{equation*}
        \xymatrix{
            X_i \ar@{^{(}->}[r] \ar@{^{(}->}[d] & Y_i \ar@{^{(}->}[d] \ar@{->>}[r] & Z_i \ar@{^{(}->}[d]\\
            X_j \ar@{^{(}->}[r] \ar@{->>}[d] & Y_j \ar@{->>}[d] \ar@{->>}[r] & Z_j \ar@{->>}[d] \\
            X_j/X_i \ar@{^{(}->}[r] & Y_j/Y_i \ar@{->>}[r] & Z_j/Z_i .
        }
    \end{equation*}
    Because $X\colon J\to \D$ is an admissible relative Ind diagram, so is $K\to J\to^X\D$, and similarly for $K\to I\to^Z\D$. In particular, $X_j/X_i$ and $Z_j/Z_i$ are both objects in $\C$. Because $\C$ is closed under extensions in $\D$, we conclude that $Y_j/Y_i$ is also in $\C$, and thus that $Y\colon K\to \D$ is an admissible relative Ind diagram.
\end{proof}

\begin{lemma}[Straightening]\label{lemma:straightening}
Every morphism $f\colon X \rightarrow Y$ in $\Ind(\D,\C)$ (see Definition \ref{defi:relative_tate}) can be represented (that is, \emph{straightened}) by a colimit of morphisms in $\D$,
$(X_i \rightarrow Y_i)_{i \in I},$
where $(X_i)_{i \in I}$, $(Y_i)_{i \in I}$ are relative admissible Ind diagrams.
\end{lemma}

\begin{proof}
We choose presentations $(X_i)_{i \in K}$, $(Y_i)_{i \in K'}$ for $Y$ as a relative admissible Ind diagram. According to \cite[Lemma 3.9]{TateObjectsExactCats} there exist cofinal maps $I \rightarrow K$, $I \rightarrow K'$, such that the induced relative admissible Ind diagram $(Y_i)_{i \in I}$ fits into a diagram of morphisms
$$(X_i \rightarrow Y_i)_{i \in I}.$$
This concludes the argument.
\end{proof}

Henceforth, we omit the cardinality bound $\kappa$ from our notation.

Under favourable conditions, the exact categories of relative Ind and Pro objects $\Ind(\D,\C)$ and $\Pro(\D,\C)$ are abelian.

\begin{definition}\label{defi:noetherian}
\begin{itemize}
\item[(a)] An abelian category $\A$ is said to be \emph{Noetherian} if for every object $X \in \A$, an ascending countable sequence of subobjects
$$X_0 \subset X_1 \subset \cdots \subset X$$
eventually stabilises, that is, there exists an index $i \in \mathbb{N}$, such that $X_{j + 1} = X_j$ for all $j \geq i$. It is said to be \emph{Artinian} if every descending countable sequence of subobjects stabilises.  
\item[(b)] Let $\A_1 \subset \A_2$ be an exact, fully faithful inclusion of abelian categories. We say that the pair $(\A_2,\A_1)$ is \emph{Noetherian}, if for every object $X \in \A_2$, every countable sequence of subobjects $X_i \subset X$ as above having $X_{i+1}/X_i \in \A_1$ eventually stabilises. Analogously for \emph{Artinian}.
\end{itemize}
\end{definition}

\begin{example}\mbox{}
    \begin{enumerate}
        \item For a Noetherian commutative ring $R$, the abelian category of finitely generated $R$-modules is Noetherian. If $R$ is Artinian as a ring, the category is also an Artinian.
        \item An abelian category $\Cc$ is Noetherian if and only if $\Cc^{\op}$ is Artinian.   
    \end{enumerate}
\end{example}

\begin{proposition}\label{prop:abelian}
Let $\D$ be an abelian category, and $\C \subset \D$ a Serre subcategory. Then,
\begin{itemize}
\item[(1)] the exact category $\Ind(\D,\C)$ is equivalent to an abelian category with the maximal exact structure if and only if the pair $(\D,\C)$ is Noetherian.
\item [(2)] the exact category $\Pro(\D,\C)$ is equivalent to an abelian category with the maximal exact structure if and only if the pair $(\D,\C)$ is Artinian.
\end{itemize}
Here Noetherian and Artinian are to be understood in the sense of Definition \ref{defi:noetherian} (b).
\end{proposition}

\begin{proof}
We shall only deal with the case of $\Ind(\D,\C)$, for Pro objects it suffices to invert arrows. Freyd has shown in \cite[Prop. 3.1]{MR0209333} that an exact category is abelian if and only if every morphism $X \rightarrow Y$ is \emph{admissible}, that is, admits a factorisation
$$X \twoheadrightarrow I \hookrightarrow Y,$$
where the first morphism is an admissible epimorphism, and the second one an admissible monomorphism.

For a morphism $X \xrightarrow{f} Y$ in $\Indk(\D,\C) \subset \Indk(\D)$ we may choose a straightening by Lemma \ref{lemma:straightening}, that is, a presentation as a map of diagrams $(f_i \colon X_i \rightarrow Y_i)_{i \in I}$, where each $(X_i)_{i \in I}$ and $(Y_i)_{i \in I}$ is a relative admissible Ind diagram in $\D$. In particular, we may assume that the transition maps $X_i \rightarrow X_j$ and $Y_i \rightarrow Y_j$ are monomorphisms with quotients in $\D$.

Since $\D$ is abelian, we obtain a factorisation $X_i \twoheadrightarrow I_i \hookrightarrow Y_i$, where each $I_i$ is the image of the morphism $f_i$. The equivalence
$$I_i \cong \coker(\ker f_i \rightarrow X_i)$$
implies that we have a commuting diagram with exact rows
\[
\xymatrix{
0 \ar[r] & \ker f_i \ar[r] \ar[d] & X_i \ar[r] \ar[d] & I_i \ar@{-->}[d] \ar[r] & 0\\
0 \ar[r] & \ker f_j \ar[r] & X_j \ar[r] & I_j \ar[r] & 0
}
\]
for every pair of indices $i \leq j$, and in particular obtain canonical maps $I_i \rightarrow I_j$. The commutative square
\[
\xymatrix{
I_i \ar[r] \ar[d] & Y_i \ar[d] \\
I_j \ar[r] & Y_j
}
\]
implies that $I_i \hookrightarrow I_j$ is an admissible monomorphism. Moreover we see that $I_j/I_i$ is a quotient of $X_j/X_i$, thus belongs to $\C$, since $\C$ is a Serre subcategory. We denote by $I = \colim_{i \in I} I_i$ the corresponding object of $\Ind(\D,\C)$.

So far we have produced a factorisation $X \rightarrow I \rightarrow Y$. The morphism $X \rightarrow I_i$ is an admissible epimorphism, since it fits into an exact sequence given by the colimit of
$$0 \rightarrow \ker f_i \rightarrow X_i \rightarrow I_i \rightarrow 0.$$
In order to conclude the proof, we need to show that $I \rightarrow Y$ is an admissible monomorphism in $\Ind(\D,\C)$. In $\Lex \D$ it fits into a short exact sequence given by the colimit of
$$0 \rightarrow I_i \rightarrow Y_i \rightarrow Y_i/I_i \rightarrow 0,$$
but the direct system $(Y_i/I_i)_{i \in I}$ is not necessarily admissible. However, according to Lemma \ref{lemma:trick} it admits a presentation by a relative admissible Ind diagram, which concludes the proof of one direction. This is where the relative Noetherian assumption is key.

Vice versa, assume that $(\D,\C)$ is not Noetherian, that is, there exists an object $Y \in \D$ and a sequence of subobjects
$$X_i \subset Y,$$
such that $X_i \neq X_{i + 1}$, and $X_{i+1}/X_i \in \C$. Let us denote by $X = \colim_{i \in I}X_i$ the corresponding object of $\Ind(\D,\C)$. We have a morphism $X \xrightarrow{f} Y$, induced by the inclusions $X_i \hookrightarrow Y$. We claim that $f$ does not have a cokernel in $\Ind(\D,\C)$. Indeed, assume that $f$ has a cokernel, which implies that $f$ is an admissible monomorphism in $\Ind(\D,\C)$. Hence, $f$ is also an admissible monomorphism in $\Ind(\D)$. However, \cite[Proposition 3.10]{TateObjectsExactCats} shows that $\D \subset \Ind(\D)$ is left $s$-filtering. This implies that every $X \in \Ind(\D)$, such that $X \subset Y$, we have in fact $X \in \D$. However, the object $X$ we have constructed above cannot belong to $\D$ because the sequence
$$\cdots \subset X_i \subset X_{i+1} \subset \cdots$$
does not stabilise.  If $X$ were in $\D$, then the isomorphism $X \rightarrow \colim_i X_i$ would factor through a fixed $X_i$, in which case, we would have an epic admissible monic $X_i\to X$, and thus $X$ would be isomorphic to $X_i$. We conclude that $\Ind(\D,\C)$ cannot be abelian, if $(\D,\C)$ is not relatively Noetherian.
\end{proof}

\begin{lemma}\label{lemma:trick}
Let $\D$ be an abelian category with a Serre subcategory $\C \subset \D$, such that the pair $(\D,\C)$ satisfies the relative Noetherian condition of Definition \ref{defi:noetherian}. Under these assumptions, every colimit in $\Lex \D$
$$X = \colim_{i \in I} X_i,$$
where $I$ is a directed poset, and $\ker (X_i \rightarrow X_j),\text{ }\coker (X_i \rightarrow X_j) \in \C$ for every ordered pair of indices $i \leq j$, is equivalent to an object in $\Ind(\D,\C)$.
\end{lemma}

\begin{proof}
If $|I|$ is finite, the assertion follows, since $X \cong X_{\max(I)}$. In order to verify the claim for infinite $I$ we will produce a relatively admissible Ind system $(Y_i)_{i \in I}$, satisfying $\colim_{i \in I} Y_i \cong X$, such that for $i \leq j$ we have $Y_i \hookrightarrow Y_j$ is a monomorphism with $Y_j/Y_i \in \C$. We will construct $Y_i = X_i/K_i$ as a quotient of $X_i$.

For $i \leq j$ we denote by $K_{ij}$ the kernel of $X_i \rightarrow X_j$. For $j \geq k$ we have $K_{ij} \subset K_{ik} \subset X_i$. Moreover, the quotient $K_{ik}/K_{ij}$ is a quotient of $K_{ik} \in \C$, and hence itself in $\C$. We conclude that the filtered poset of subobjects $K_{ij} \subset X_i$ must stabilise at a subobject $K_i \subset X_i$.

Note that for every $i_1 \leq i_2$ we have an induced map $K_{i_1} \rightarrow K_{i_2}$: indeed, there exist indices $j_1$, $j_2$, such that $K_{i_1} = \ker(X_{i_1} \rightarrow X_j)$ for every $j \geq i_1$, and similarly for $i_2$. Hence, we may choose $j \geq i_1,i_2$, and the universal property of kernels implies the existence of a unique map $K_{i_1} \rightarrow K_{i_2}$ as in the diagram
\[
\xymatrix{
0 \ar[r] & K_{i_1} \ar[r] \ar@{-->}[d] & X_{i_1} \ar[r] \ar[d] & X_j \ar[d] \\
0 \ar[r] & K_{i_2} \ar[r] & X_{i_2} \ar[r] & X_j.
}
\]
We denote the colimit $\colim_{i \in I} K_i \in \Lex \D$ by $K$; and similarly define $Y = \colim_{i \in I} X_i/K_i$. By exactness of colimits for Grothendieck abelian categories, we have a short exact sequence
$$0 \rightarrow K \rightarrow X \rightarrow Y \rightarrow 0.$$
However, one sees easily that $K = 0$, since by construction we know that for every $i \in I$, there exists $j \geq i$, such that $K_i \rightarrow K_j$ is the zero map. In order to conclude the proof, we have to show that for every $i_1 \leq i_2$ the induced map $Y_{i_1} \rightarrow Y_{i_2}$ is a monomorphism.

As above we may choose an index $j \geq i_1,i_2$, such that $K_{i_k} = \ker(X_{i_k} \rightarrow X_j)$ for $k = 1,2$. In particular, we may identify $Y_{i_k}$ with the image $X_{i_k}/K_{i_k} \subset X_j$, and the map $Y_{i_1} \rightarrow Y_{i_2}$ with the induced inclusion of images.
\end{proof}

\begin{example}(Non-abelian admissible Ind category)
Let $\mathsf{Vect}_{f}$ be the category of finite-dimensional $k$-vector
spaces and $\mathsf{Vect}$ the category of all $k$-vector spaces. There is a
morphism%
\[
\left.
{\textstyle\bigoplus\limits_{i=0\ldots\infty}^{\Ind(\mathsf{Vect})}}
\right.  k\overset{f}{\longrightarrow}\left.
{\textstyle\bigoplus\limits_{i=0\ldots\infty}^{\mathsf{Vect}}}
\right.  k\text{,}%
\]
where the first coproduct is formed in $\Ind(\mathsf{Vect}%
)$, and the latter in $\mathsf{Vect}$ itself. In terms of an admissible Ind diagram,
$f$ is $(i\mapsto(k^{\oplus i}\hookrightarrow k^{\oplus\infty}))$. This
morphism does not possess a cokernel, so that $\Ind(\mathsf{Vect})$ cannot be an abelian category. This example also
re-affirms that the inclusion $\C\hookrightarrow\Ind(\C)$
does not preserve colimits. If we work instead with the full Ind category $\mathsf{Ind}(\C)$, i.e. we allow also
Ind diagrams whose transition morphisms are not monics, $\mathsf{Ind}(\C)$ is always abelian once $\C$ is, and our
$f$ would permit a cokernel. Its transition morphisms would all be epics, so it is non-admissible.
\end{example}

\subsection{Further Properties}
The next lemma is a slight generalization of \cite[Proposition 3.14]{TateObjectsExactCats}; mutatis mutandi the proof is the same.
\begin{lemma}\label{lemma:SrelInd}
    For $k\ge 0$, there exist canonical equivalences
    \begin{align*}
        \Ind(S_k\D,S_k\C)&\to^\simeq S_k(\Ind(\D,\C)),\\
        \Pro(S_k\D,S_k\C)&\to^\simeq S_k(\Pro(\D,\C)),\\
        \elTate(S_k\D,S_k\C)&\to^\simeq S_k\elTate(\D,\C).
    \end{align*}
\end{lemma}

In Definition \ref{defi:elTate} we defined elementary Tate objects as the full subcategory of admissible Ind Pro objects which possess a lattice. The concept of lattices also exists for relative elementary Tate objects and is of equal importance.

\begin{definition}\label{defi:relative_gr}\mbox{}
    \begin{enumerate}
        \item For $V \in \elTate(\D,\C)$ we say that an admissible monic $L \hookrightarrow V$ is a \emph{relative lattice}, if $L \in \Pro(\D)$, and $V/L \in \Ind(\C)$.
        \item For two relative lattices $L,L'$, we say that $L \leq L'$, if the inclusion $L \hookrightarrow V$ factors through $L' \hookrightarrow V$ via an admissible monic $L \hookrightarrow L'$ (by Lemma \ref{lemma:reflected}).
        \item Define the \emph{relative Sato Grassmannian} $\Gr_{\C}(V)$ to be the partially ordered set of relative lattices of $V$.
    \end{enumerate}
\end{definition}

It follows directly from the definition that every elementary relative Tate object has a lattice. In fact, the existence of lattices characterises elementary relative Tate objects.

\begin{lemma}
For $V \in \Ind \Pro(\D)$ the following assertions are equivalent:
\begin{itemize}
\item[(a)] We have $V \in \elTate(\D,\C)$.
\item[(b)] There exists a relative lattice, that is, a short exact sequence $L \hookrightarrow V \twoheadrightarrow V/L$ with $L \in \Pro(\D)$ and $V/L \in \Ind(\C)$.
\end{itemize}
\end{lemma}

\begin{proof}
We have seen in Lemma \ref{lemma:relative_exact} that $\elTate(\D,\C) \subset \Ind \Pro(\D)$ is an extension-closed subcategory. Since $\Pro(\D) \subset \elTate(\D,\C) \supset \Ind(\C)$, we see that (b) implies (a).

Conversely, if $V$ is in $\elTate(\D,\C)$, there exists a presentation as $\colim_{i \in I} V_i$, where $(V_i)_{i \in I}$ is a directed system in $\Pro(\D)$, with $V_j/V_i \in \C$ for $j \geq i$. Lemma \ref{lemma:ind_mono} implies that $V_i \hookrightarrow V$ is an admissible monomorphism in $\Ind \Pro(\D)$. The quotient is given by $\colim_{j \geq i} V_j/V_i \in \Ind(\C)$. This shows that $V$ possesses a relative lattice.
\end{proof}

\begin{corollary}
The smallest extension-closed subcategory of $\Ind \Pro(\D)$ which contains $\Pro(\D)$ and $\Ind(\C)$ is $\elTate(\D,\C)$.
\end{corollary}

We now record the key property of relative lattices.

\begin{proposition}\label{prop:rel_Gr-directed}
    If $\D$ is idempotent complete, then the relative Sato Grassmannian $\Gr_{\C}(V)$ is a directed partially ordered set.
\end{proposition}
\begin{proof}
    Abuse notation and let $V\colon I\to\Pro(\D)$ be an admissible relative Tate diagram representing $V$. Note that, by definition, for all $i\in I$, $V_i\in\Gr_{\C}(V)$. Now let $L_1,L_2\in\Gr_{\C}(V)$. Because $\Pro(\D)$ is left filtering in $\elTate(\D,\C)$, there exists $i\in I$ such that we have a commuting triangle
    \begin{equation*}
        \xymatrix{
          L_1\oplus L_2 \ar[rr] \ar[dr]  &  &    V_i \ar@{^{(}->}[dl]    \\
          & V                 }
    \end{equation*}
    in $\elTate(\D,\C)$. If $\D$ is idempotent complete, then Lemma 6.9 of \cite{TateObjectsExactCats} shows that for $a=1,2$ the map $L_a\to V_i$ is an admissible monic in $\Pro(\D)$.
\end{proof}

\begin{lemma}\label{lemma:relGrinC}
    Let $\C\subset\D$ be a left or right s-filtering subcategory. Let $V\in\elTate(\D,\C)$ and let $L_1\le L_2\in\Gr_{\C}(V)$. Then $L_2/L_1\in\C$.
\end{lemma}
\begin{proof}
    By \cite[Proposition 6.6]{TateObjectsExactCats}, we know that $L_2/L_1\in\D$. Lemma \ref{lemma:IndCinIndD} and Noether's Lemma show that $L_2/L_1\in\Ind(\C)$. By Lemma \ref{lemma:IndCD}, we have $L_2/L_1\in\C$.
\end{proof}

\begin{lemma}\label{lemma:DicCic}
    Let $\D$ be idempotent complete, and let $\C\subset\D$ be a left or right filtering subcategory. Then $\C$ is idempotent complete.
\end{lemma}
\begin{proof}
By categorical duality, we may reduce the claim to the assumption that $\C$ is left filtering in $\D$. Let $X\to^p X$ be an idempotent in $\C$. Because $\D$ is idempotent complete, there exists $Y\in\D$ such that $Y=\ker(p)$. Because $\C \subset \D$ is left filtering, $\C$ is closed under subobjects in $\D$. Hence, $Y \hookrightarrow X$ also belongs to $\C$, which implies that $\ker(p) \in \C$.
\end{proof}

As a next step, we investigate the filtering properties of the inclusions $\D \hookrightarrow \Ind(\D,\C)$, $\Pro(\D) \hookrightarrow \elTate(\D,\C)$, and study the relation between the categorical quotients.

\begin{proposition}\label{prop:z2}
    Let $\C\subset\D$ be a subcategory which is closed under extensions. The inclusions $\D \hookrightarrow \Ind(\D,\C)$, and $\Pro(\D) \hookrightarrow \elTate(\D,\C)$ are left $s$-filtering. The inclusions $\Ind(\C)\to\Ind(\D,\C)$ and $\Ind(\D,\C)\to\elTate(\D,\C)$ induce exact equivalences
    \begin{equation}\label{eqn:rel_ind_quot}
        \Ind(\C)/\C \to^\simeq \Ind(\D,\C)/\D \to^\simeq \elTate(\D,\C)/\Pro(\D).
    \end{equation}
\end{proposition}
\begin{proof}
    We first show the inclusions are left s-filtering.  The second inclusion is a special case of the first. For the first, we observe that $\D$ is left special in $\Ind(\D,\C)$ because $\D$ is left special in $\Ind(\D)$ \cite[Lemma 2.18]{TateObjectsExactCats}. Further, $\D$ is left filtering in $\Ind(\D,\C)$ for the same reason it is in $\Ind(\D)$, namely given any admissible relative Ind diagram $Y\colon I\to \D$, and given any
    \begin{equation*}
        X\to \widehat{Y}
    \end{equation*}
    there exists $i\in I$ such that $X$ factors through the admissible monic $Y_i\hookrightarrow\widehat{Y}$.

    We now establish the equivalences. The inverse equivalence $$\Ind(\D,\C)/\D\to\Ind(\C)/\C$$ is defined as follows. Let $\Dir_\ast(\D,\C)\subset\Dir(\D)$ be the full subcategory of relative admissible Ind diagrams indexed by directed partially ordered sets with an initial object. A slight modification of the proof of \cite[Proposition 5.11]{TateObjectsExactCats} shows that $\Ind(\D,\C)$ is equivalent to the localization $\Dir_\ast(\D,\C)[W^{-1}]$, where $W$ is the subcategory of final maps as in \cite[Proposition 3.15]{TateObjectsExactCats}. By inspection, the assignment
    \begin{equation*}
        (I \xrightarrow{X} \D) \mapsto (I \xrightarrow{X/X_0} \C)
    \end{equation*}
    induces a functor $\Dir_\ast(\D,\C)\to\Ind(\C)$, such that the induced functor $$\Dir_\ast(\D,\C)\to\Ind(\C)/\C$$ factors through the localization $\Dir_\ast(\D,\C)\to\Ind(\D,\C)$. By inspection, this functor is the desired inverse for the first functor of \eqref{eqn:rel_ind_quot}. To see that it is exact, apply the straightening construction for exact sequences (cf. \cite[Proposition 3.12]{TateObjectsExactCats}).

    Mutatis mutandi, the proof of Proposition 5.28 of \cite{TateObjectsExactCats} defines an exact functor
    \begin{equation*}
        \elTate(\D,\C)\to\Ind(\C)/\C
    \end{equation*}
    which factors through the localization $\elTate(\D,\C)/\Pro(\D)$. By inspection, this functor is inverse to the canonical map $\Ind(\C)/\C\to\elTate(\D,\C)/\Pro(\D)$.

    The previous lemma now combines with the 2 of 3 property for equivalences to imply that the second functor in \eqref{eqn:rel_ind_quot} is an equivalence as well.
\end{proof}

\begin{proposition}\label{prop:relTatequot}
    Let $\D$ be idempotent complete and let $\C\subset\D$ be a right s-filtering subcategory. Then $\C\subset\Pro(\D)$ is right s-filtering, and the inclusion $\Pro(\D)\subset\elTate(\D,\C)$ induces an exact functor between exact categories
    \begin{equation*}
        \Pro(\D)\to\elTate(\D,\C)/\Ind(\C),
    \end{equation*}
    yielding an equivalence of categories $\Pro(\D)/\C \cong \elTate(\D,\C)/\Ind(\C)$.
\end{proposition}
\begin{proof}
    The definition of right s-filtering implies that a composition of right s-filtering embeddings is again right s-filtering. Therefore, our assumption on $\C$ together with the fact that $\D\hookrightarrow\Pro(\D)$ is right s-filtering (cf. \cite[Theorem 4.2(2)]{TateObjectsExactCats}) implies that $\C\subset\Pro(\D)$ is right s-filtering.

Using Lemma \ref{lemma:relGrinC}, the same argument as for \cite[Proposition 5.29]{TateObjectsExactCats} shows that the assignment
    \begin{equation*}
        V\mapsto L
    \end{equation*}
    (sending a relative Tate object to a relative lattice) extends to an exact functor
    \begin{equation}\label{relTatequot}
        \elTate(\D,\C)\to\Pro(\D)/\C.
    \end{equation}
    To see that this factors through $\elTate(\D,\C)/\Ind(\C)$, let
    \begin{equation*}
        V_0\hookrightarrow V_1\twoheadrightarrow Z
    \end{equation*}
    be a short exact sequence of relative Tate objects with $Z\in\Ind(\C)$. By the universal property of localizations, it suffices to show that \eqref{relTatequot} sends the map $V_0\hookrightarrow V_1$ to an isomorphism in $\Pro(\D)/\C$.

    To check this, we let $L_0\hookrightarrow V_0$ be a relative lattice. By the definition of morphisms in $\elTate(\D,\C)$, the inclusion
    \begin{equation*}
        L_0\hookrightarrow V_1
    \end{equation*}
    factors through a relative lattice $L_1\hookrightarrow V_1$. Therefore, the functor \eqref{relTatequot} sends the map $V_0\hookrightarrow V_1$ to $L_0\to L_1$. We claim that this map is an isomorphism in $\Pro(\D)/\C$, i.e. that it is an admissible monic in $\Pro(\D)$ with cokernel in $\C$.

    By Lemma \ref{lemma:relGrinC}, it suffices to show that the admissible monic $L_0\hookrightarrow V_1$ is a relative lattice. This follows from Noether's lemma and Lemma \ref{lemma:IndCinIndD}. Indeed, we have a short exact sequence in $\elTate(\D,\C)$
    \begin{equation*}
        V_0/L_0\hookrightarrow V_1/L_0\twoheadrightarrow V_1/V_0.
    \end{equation*}
    By assumption $V_0/L_0$ and $V_1/V_0$ are both in $\Ind(\C)$. Therefore $V_1/L_0$ is as well.

    We have shown that \eqref{relTatequot} induces an exact functor
    \begin{equation*}
        \elTate(\D,\C)/\Ind(\C)\to\Pro(\D)/\C.
    \end{equation*}
    From the definitions, we see that this is an inverse to the map $$\Pro(\D)/\C\to\elTate(\D,\C)/\Ind(\C).$$
    This concludes the proof.
\end{proof}

\subsection{Examples}
Let $X$ be a Noetherian scheme, and let $Z\subset X$ be a closed subscheme. Denote by $j\colon X\setminus Z\hookrightarrow X$ the inclusion of the complement of $Z$. Denote by $\Coh_Z(X)$ the full subcategory of $\Coh(X)$ consisting of coherent sheaves with set-theoretic support in $Z$. Denote by $\QCoh(X,\Coh(X\setminus Z))$ the full subcategory of $\QCoh(X)$ consisting of quasi-coherent sheaves whose restriction to $X\setminus Z$ is coherent.

\begin{proposition}\label{prop:example}
    There exists an exact equivalence
    \begin{equation*}
        \QCoh(X,\Coh(X\setminus Z))\to^\simeq \Ind(\Coh(X),\Coh_Z(X)),
    \end{equation*}
    and this equivalence fits into a 2-commuting square
    \begin{equation*}
        \xymatrix{
            \QCoh(X,\Coh(X\setminus Z)) \ar[rr] \ar[d]_\simeq && \QCoh(X) \ar[d]^\simeq\\
            \Ind(\Coh(X),\Coh_Z(X)) \ar[rr] && \Ind(\Coh(X)).
        }
    \end{equation*}
\end{proposition}
\begin{proof}
    Recall that, because $X$ is Noetherian, there is an exact equivalence $\QCoh(X)\simeq \Ind(\Coh(X))$ which sends a quasi-coherent sheaf $F$ to the Ind object represented by the admissible Ind diagram of coherent subsheaves of $F$ (see \cite[Tag 01PG]{stacks-project}).

    Because $\Ind(\Coh(X),\Coh_Z(X))$ is a fully exact subcategory of $\Ind(\Coh(X))$, it suffices to show that a quasi-coherent sheaf is in $\Ind(\Coh(X),\Coh_Z(X))$ if and only if its pullback to $X\setminus Z$ is coherent.  The ``only if'' is clear.

    Let $F$ be represented by an Ind diagram $F\colon I\to \Coh(X)$. Suppose the pullback $j^\ast F$ is coherent. Then there exists a final subdiagram $J\subset I$ such that the diagram
    \begin{equation*}
        J\hookrightarrow I\to^F\Coh(X)\to^{j^\ast}\Coh(X\setminus Z)
    \end{equation*}
    is isomorphic to a constant diagram. In particular, for all $j<k$ in $J$, the cokernel $F_k/F_j$ has set-theoretic support in $Z$. We conclude that $\QCoh(X,\Coh(X\setminus Z))\subset\Ind(\Coh(X),\Coh_Z(X))$.
\end{proof}

\begin{corollary}
    If the inclusion $j$ is affine, there exists a 2-commuting diagram of exact functors
    \begin{equation*}
        \xymatrix{
            \Coh(X\setminus Z) \ar[rr]^{j_{\ast}} \ar[d] && \QCoh(X) \ar[d]^\simeq\\
            \Ind(\Coh(X),\Coh_Z(X)) \ar[rr] && \Ind(\Coh(X))
        }.
    \end{equation*}
\end{corollary}
\begin{proof}
    Because $X\setminus Z \subset X$ is affine, the push-forward $j_\ast$ gives an exact functor
    \begin{equation*}
        j_\ast\colon\Coh(X\setminus Z)\to \QCoh(X).
    \end{equation*}
    Because the co-unit of the adjunction $j_\ast\dashv j^\ast$ is an isomorphism, we see that $j_\ast$ factors through $\QCoh(X,\Coh(X\setminus Z))$.
\end{proof}

\begin{proposition}\label{prop:C_Z}
    There exists an exact functor
    \begin{equation*}
        \mathbf{C}_Z\colon\Coh(X)\to\Pro(\Coh_Z(X)).
    \end{equation*}
\end{proposition}
\begin{proof}
    For all $r\ge 1$, let $j_r\colon Z^r\to X$ denote the inclusion of the $r^{th}$-order formal neighborhood of $Z$ in $X$. For $F\in\Coh(X)$, define
    \begin{equation*}
        \mathbf{C}_Z(F):=\lim_r j_{r,\ast}j_r^\ast F
    \end{equation*}
    By inspection, the transition maps $j_{r,\ast}j_r^\ast F\to j_{r-1,\ast}j_{r-1}^\ast F$ are epimorphisms in the abelian category $\Coh_Z(X)$. Therefore, the assignment $F\mapsto\mathbf{C}_Z(F)$ defines a functor
    \begin{equation*}
        \mathbf{C}_Z\colon\Coh(X)\to\Pro(\Coh_Z(X)).
    \end{equation*}
    By the Artin--Rees Lemma (e.g. \cite[Proposition 10.12]{MR0242802}), this functor is exact.
\end{proof}

\begin{corollary}\label{cor:T_Z}
    There exists a $2$-commuting diagram of functors
    \begin{equation*}
        \xymatrix{
            \Coh_Z(X) \ar[rr] \ar[d]^1 && \Coh(X) \ar[rr] \ar[d]^{\mathbf{C}_Z} && \QCoh(X,\Coh(X\setminus Z)) \ar[d]^{\mathbf{T}_Z}\\
            \Coh_Z(X) \ar[rr] && \Pro(\Coh_Z(X)) \ar[rr] && \elTate(\Coh_Z(X)),
        }
    \end{equation*}
    moreover the functor $\mathbf{T}_Z$ satisfies the property that the composition
    $$\Coh(U) \xrightarrow{\mathbf{T}_Z} \elTate(\Coh_Z(X)) \rightarrow \elTate(\Coh_Z(X))/\Ind(\Coh_Z(X)) \cong \Pro(\Coh_Z(X))/\Coh_Z(X)$$
    is an exact functor.
\end{corollary}

\begin{proof}
    The definition of $\mathbf{C}_Z$ ensures that if $F\in\Coh_Z(X)$, then the Pro object $\mathbf{C}_Z(F)$ is represented by the constant Pro diagram on $F$. This accounts for the left square. For the right square, we observe that $\mathbf{C}_Z$ gives an exact functor of pairs
    \begin{equation*}
        (\Coh(X),\Coh_Z(X))\to^{\mathbf{C}_Z}(\Pro(\Coh_Z(X)),\Coh_Z(X)).
    \end{equation*}
    The exact functor $\mathbf{T}_Z$ is the corresponding map
    \begin{align*}
        \QCoh(X,\Coh(X\setminus Z))\to^{\simeq} \Ind(\Coh(X),\Coh_Z(X))\to &\Ind(\Pro(\Coh_Z(X)),\Coh_Z(X))\\
        &=:\elTate(\Coh_Z(X)).
    \end{align*}
    This concludes the construction and the proof of the first part. The second part follows from Lemma \ref{lemma:UtoPro} below, and the fact that the exact functor $$\elTate(\Coh_Z(X))/\Ind(\Coh_Z(X)) \cong \Pro(\Coh_Z(X))/\Coh_Z(X)$$
 is an equivalence of exact categories by \cite[Proposition 5.34]{TateObjectsExactCats}.
\end{proof}

\begin{lemma}\label{lemma:UtoPro}
Let $X$ be a Noetherian affine scheme, and $Z$ a closed subset, we denote by $j\colon U \hookrightarrow X$ the inclusion of the open complement. There exists an exact functor $$\mathbf{B}_Z\colon\Coh(U) \rightarrow \Pro(\Coh_Z(X))/\Coh_Z(X),$$
rendering the diagram of categories
\[
\xymatrix{
\Coh(X) \ar[r]^{j^*} \ar[d] & \Coh(U) \ar[d]^{\mathbf{B}_Z} \\
\Pro(\Coh_Z(X)) \ar[r] & \Pro(\Coh_Z(X))/\Coh_Z(X)
}
\]
$2$-commutative.
\end{lemma}
\begin{proof}
For every coherent sheaf $\Fc \in \Coh(U)$ we denote by $K_{\Fc}$ the directed set of coherent subobjects $N$ of $j_*\Fc$ satisfying $N|_U = \Fc$. We have $$j_*\Fc \cong \colim_{N \in K_{\Fc}}N.$$ As a consequence we obtain that for every morphism $\Fc_1 \rightarrow \Fc_2$, and every $N_1 \in K_{\Fc_1}$, there exists an $N_2 \in K_{\Fc_2}$, such that $N_1 \hookrightarrow \Fc_1 \rightarrow \Fc_2$ factors through $N_2$.

We abuse notation and denote by $$\mathbf{B}_Z\colon \Coh(X) \rightarrow \Pro(\Coh_Z(X))/\Coh_Z(X)$$ the functor given by the composition of $\mathbf{C}_Z$ and $$\Pro(\Coh_Z(X)) \rightarrow \Pro(\Coh_Z(X))/\Coh_Z(X).$$
Now we define $$\mathbf{C}_Z\colon \Coh(U) \rightarrow \Pro(\Coh_Z(X))/\Coh_Z(X)$$ as the colimit $$\mathbf{B}_Z(\Fc) = \colim_{N \in K_{\Fc}}\mathbf{C}_Z(N).$$
We observe that for $N \subset N'$, the induced bonding map $\mathbf{C}_Z(N) \rightarrow \mathbf{C}_Z(N')$ is an isomorphism in $\Pro(\Coh_Z(X))/\Coh_Z(X)$, since the kernel is given by $\mathbf{C}_Z(N'/N) \in \Coh_Z(Y)$. Therefore the colimit exists.

To conclude the proof we have to show that the functor we have defined is exact. Let
\begin{equation}\label{eqn:f123}0 \rightarrow \Fc_1 \rightarrow \Fc_2 \rightarrow \Fc_3 \rightarrow 0\end{equation}
be a short exact sequence in $\Coh(U)$. We obtain an exact sequence
$$0 \rightarrow j_* \Fc_1 \rightarrow j_* \Fc_2 \rightarrow j_* \Fc_3 \rightarrow R^1j_* \Fc_1.$$
The quasi-coherent sheaf $\mathcal{G}_3 := \ker(j_* \Fc_3 \rightarrow R^1j_* \Fc_1)$ agrees with $\Fc_3$ after restriction to $U$, since $R^1j_*\Fc|_U = 0$. Choose a coherent subsheaf $\bar{\Fc}_3 \subset \mathcal{G}_3$, satisfying $\bar{\Fc}_3|_U = \Fc_3$. By what we remarked above, there exists a coherent subsheaf $\bar{\Fc}_2 \subset j_*\Fc_2$, satisfying $\bar{\Fc}_2 |_U= \Fc_2|_U$, mapping to $\bar{\Fc}_3$. We define the kernel of this map to be $\bar{\Fc}_1$. We have a commutative diagram with commutative rows
\[
\xymatrix{
0 \ar[r] \ar[d] & \bar{\Fc}_1 \ar[r] \ar[d] & \bar{\Fc}_2 \ar[r] \ar[d] & \bar{\Fc}_3 \ar[r] \ar[d] & 0 \\
0 \ar[r] & j_*\Fc_1 \ar[r] & j_*\Fc_2 \ar[r] & \mathcal{G}_3 \ar[r] & 0.
}
\]
Applying the restriction functor $j^*$, we see that the two vertical maps on the right become isomorphisms, hence the same is true for $\bar{\Fc_1} \rightarrow j_*\Fc_1$. After applying $\mathbf{C}_Z$, the upper short exact sequence can be identified with the sequence obtained by applying the functor $\mathbf{B}_Z$ to Equation \ref{eqn:f123}. This concludes the proof of the assertion.
\end{proof}

\begin{corollary}
    Let $W\subset Z\subset X$ be a chain of closed subschemes. There exists a 2-commuting diagram of exact maps
    \begin{equation*}
        \xymatrix{
            \Coh_W(X) \ar[rr] \ar[d]^1 && \Coh(X) \ar[rr] \ar[d]^{\mathbf{C}_Z} && \QCoh(X,\Coh(X\setminus W))\ar[d]^{\mathbf{T}_{Z,W}}\\
            \Coh_W(X) \ar[rr] && \Pro(\Coh_Z(X)) \ar[rr] && \elTate(\Coh_Z(X),\Coh_W(X))/\Ind(\Coh_Z(X)).
        }
    \end{equation*}
\end{corollary}
\begin{proof}
    The proof is the same as for the previous corollary, once we observe that, for $W\subset Z$ and $F\in\Coh_W(X)$, we have that $\mathbf{C}_Z(F)$ is represented by the constant Pro diagram on $F$.
\end{proof}

We conclude this section with a brief discussion of another exact functor of geometric origin, which takes values in a category of Pro objects.

\begin{oldexample}[Deligne]\label{ex:deligne}\normalfont
Let $X$ be a Noetherian scheme and $j:U\hookrightarrow X$ an open
immersion. Let $\mathcal{J}$ be the ideal sheaf of the closed complement. It
does not matter whether we take the reduced ideal sheaf or some
nil-thickening. Deligne \cite{deligne} defines the following: For any coherent
sheaf $\mathcal{F}$ on $U$ let $\overline{\mathcal{F}}$ be a \textit{coherent}
prolongment to $X$. Such an $\overline{\mathcal{F}}$ always exists, but it is
not canonical. A concrete construction is given by considering the poset of
coherent subsheaves $\mathcal{F}_{i}\subseteq j_{\ast}\mathcal{F}$ of the
quasi-coherent sheaf $j_{\ast}\mathcal{F}$. The restrictions $j^{\ast
}\mathcal{F}_{i}$ will become stationary and equal $\mathcal{F}$, and once
$\mathcal{F}_{i}$ is chosen so that this occurs, $\mathcal{F}_{i}$ is a
feasible candidate for $\overline{\mathcal{F}}$. We get a Pro-diagram%
\begin{equation}
\mathcal{J}^{i}\cdot\overline{\mathcal{F}}\longleftarrow\mathcal{J}%
^{i+1}\cdot\overline{\mathcal{F}}\longleftarrow\mathcal{J}^{i+2}%
\cdot\overline{\mathcal{F}}\longleftarrow\cdots\text{,}\label{line1}%
\end{equation}
which defines an object in $\mathsf{Pro\,}\mathrm{Coh}(X)$. Deligne shows that
this defines a functor%
\[
j_{!}:\mathrm{Coh}(U)\longrightarrow\mathsf{Pro\,}\mathrm{Coh}(X)\text{,}%
\]
which is essentially left adjoint to the pull-back $j^{\ast}$:%
\[
\operatorname*{Hom}\nolimits_{U}(\mathcal{F},j^{\ast}\mathcal{G})\overset
{\sim}{\longrightarrow}\operatorname*{Hom}\nolimits_{X}(j_{!}\mathcal{F}%
,\mathcal{G})
\]
holds functorially for all $\mathcal{F}$ coherent on $U$ and $\mathcal{G}$
quasi-coherent on $X$. In favourable situations, if $\overline{\mathcal{F}}$
has no torsion, the Pro-diagram in Equation \ref{line1} has monic transition
maps. It describes an intersection, the \textquotedblleft$\mathcal{J}%
$-divisible\textquotedblright\ sections of $\overline{\mathcal{F}}$. This
makes this Pro-system very different from an admissible Pro-diagram with epic
transition maps. Moreover, unless $X$ is an Artinian scheme, Equation
\ref{line1} cannot be replaced by an admissible system, Proposition \ref{prop:abelian}
does not apply since $\mathrm{Coh}(X)$ is not an Artinian abelian category.
Being exact, $j_{!}$ induces a map in $K$-theory, but by the Eilenberg
swindle, $j_{!}:K_{\mathrm{Coh}(U)}\rightarrow K_{\mathsf{Pro\,}%
\mathrm{Coh}(X)}$ is just a map to the zero spectrum, so there is nothing
interesting to see in $K$-theory anyway. Nonetheless, $j_{!}$ is of
course an important functor for other purposes:\ Deligne uses $j_{!}$ as an
ingredient to define a derived push-forward \textquotedblleft with compact
supports\textquotedblright\ $\mathbf{R}f_{!}$ for (non-admissible
Pro-)coherent sheaves. Classically, both $\mathbf{R}f_{!}\leftrightarrows
\mathbf{R}f^{!}$ were only defined for proper morphisms, but this trick
allowed him to devise a generalization to compactifiable morphisms. See
\cite{deligne} for more.
\end{oldexample}

Deligne's functor is related to $\mathbf{T}_Z$ by a short exact sequence
\begin{equation}\label{eqn:deligne}
j_{!}(-) \hookrightarrow j_*(-) \twoheadrightarrow \mathbf{T}_Z(-)
\end{equation}
of (not necessarily exact) functors from $\Coh(U)$ to $\mathsf{Ind}\;\mathsf{Pro} \Coh(X)$. Indeed, for $\F \in \Coh(U)$ we define $j_!\F$ by choosing an extension $\bar{\F}$, and forming the limit over the inverse system $\mathcal{J}^i\bar{\F}$. Hence, we have a short exact sequence
$$0 \rightarrow \lim_{i} \mathcal{J}^i\bar{\F} \rightarrow \bar{\F} \rightarrow \lim_i \bar{\F}/\mathcal{J}^i\bar{\F} \rightarrow 0,$$
in $\mathsf{Pro}(\Coh(X))$. The limit on the right hand side is by definition $\mathbf{C}_Z(\bar{\F})$. Taking the colimit over all possible choices for $\bar{\F}$, we obtain the short exact sequence
$$0 \rightarrow j_! \F \rightarrow j_* \F \rightarrow \mathbf{T}_Z(\F) \rightarrow 0$$
in $\mathsf{Ind}\;\mathsf{Pro} \Coh(X)$.

\section{The Relative Index Map}\label{relative}

In order to introduce the relative index map, and relate it to algebraic $K$-theory, we need to recall a few facts about Waldhausen's approach to algebraic $K$-theory for exact categories.

In \cite{MR2079996}, Schlichting established a fundamental ``Localization Theorem'' for the $K$-theory of exact categories. We will be mainly interested in its statement for Waldhausen's $S$-construction. This requires us to recall notation introduced by Waldhausen \cite{MR0802796}.

For an exact functor $f\colon \C \rightarrow \D$ of exact categories, we denote by $S_{\bullet}^r(f)$ the simplicial object of exact categories given by pairs
$$ \left(Y_1\hookrightarrow\cdots\hookrightarrow Y_n;X_1\hookrightarrow \cdots\hookrightarrow
        X_{n+1}\right) \in S_n\C \times S_{n+1}\D $$
together with an isomorphism
$$\phi\colon (Y_1 \hookrightarrow \cdots Y_n) \cong (X_2/X_1 \hookrightarrow \cdots \hookrightarrow X_{n+1}/X_1).$$
The face and degeneracy maps are induced by the ones for the Waldhausen $S$-construction; details can be found in \emph{loc. cit.}

For a category $\C$ we denote by $\C^{\grp}$ the groupoid obtained by discarding all non-invertible isomorphisms. Via the classifying space construction (that is, the geometric realisation of the nerve), we can fully faithfully embed the $2$-category of groupoids into the $\infty$-category of spaces.

Following Waldhausen, we denote by $K_{\C} = \Omega|S_{\bullet}\C^{\times}|$ the $K$-theory space of an exact category. The corresponding connective spectrum will be denoted by $\Kk_{\C}$.
\begin{proposition}(Schlichting \cite[Lemma 2.3]{MR2079996})\label{prop:Schlichting}
   Suppose that $\C\subset\D$ is a right $s$-filtering inclusion of an idempotent complete subcategory. Let
    \begin{equation*}
        S^r_\bullet(\C\subset\D)\to^q \D/\C
    \end{equation*}
   be the map of simplicial objects in categories, given by
    \begin{equation*}
        \left(Y_1\hookrightarrow\cdots\hookrightarrow Y_n;X_1\hookrightarrow \cdots\hookrightarrow
        X_{n+1}\right)\mapsto X_{n+1}.
    \end{equation*}
    Then, in the homotopy commuting cube of spaces, all diagonal arrows are equivalences
    \begin{equation*}
        \xymatrix@=9pt{
            |S_\bullet(\C)^\times| \ar[dr]^1 \ar[rr]^f \ar[dd] && |S_\bullet(\D)^\times| \ar'[d][dd] \ar[dr]^1 \\
            & |S_\bullet(\C)^\times| \ar[rr]^(.3){f} \ar[dd] && |S_\bullet(\D)^\times| \ar[dd]\\
            |S_\bullet S^r_\bullet(1_{\C})^\times| \ar'[r][rr] \ar[dr] && |S_\bullet S^r_\bullet(f)^\times| \ar[dr]^{|S_\bullet q|} \\
            & \ast \ar[rr] && |S_\bullet(\D/\C)^\times|. }
    \end{equation*}
\end{proposition}

From this proposition, Schlichting could deduce the following result.

\begin{theorem}[Schlichting's Localization Theorem]\label{thm:schlichting}
    Let $\C\subset\D$ be an idempotent complete subcategory, which is left or right $s$-filtering. The commutative square of spaces
    \begin{equation*}
        \xymatrix{
            K_{\C} \ar[r] \ar[d] & K_{\D} \ar[d] \\
            \ast \ar[r] & K_{\D/\C}
        }
    \end{equation*}
    is cartesian.
\end{theorem}

In \cite[Corollary 2.39]{ArticleTate2}, the authors established the following description of boundary maps in $K$-theory.

\begin{theorem}\label{thm:boundary}
    Let $\C\subset\D$ be a right $s$-filtering inclusion of an idempotent complete subcategory. Consider the map of spaces
    \begin{equation*}
        |S^r_\bullet(\C\subset\D)^\times|\to^{\delta} |S_\bullet(\C)^\times|,
    \end{equation*}
    induced by the map of simplicial objects in exact categories
    $$\left(Y_1\hookrightarrow\cdots\hookrightarrow Y_n;X_1\hookrightarrow \cdots\hookrightarrow
        X_{n+1}\right)\mapsto (Y_1 \hookrightarrow \cdots \hookrightarrow Y_{n}).$$
   It is canonically equivalent to the boundary map
    \begin{equation}\label{Schlichtingboundary}
        \Omega|S_\bullet(\D/\C)^\times|\to^\partial |S_\bullet(\C)^\times|
    \end{equation}
    associated to the localization sequence
    \begin{equation}\label{Schlichting}
        \xymatrix{
          |S_\bullet(\C)^\times| \ar[d] \ar[r] & |S_\bullet(\D)^\times| \ar[d] \\
          \ast \ar[r] & |S_\bullet(\D/\C)^\times|   }.
    \end{equation}
\end{theorem}

We now give a variant of the index map for relative Tate objects and relate it to boundary maps in algebraic $K$-theory.

\begin{definition}
    Let $\C \subset \D$ be an extension-closed subcategory.
    \begin{enumerate}
        \item For $n\ge 0$, define $\Gr_n^\le(\D,\C)$ to be the full subcategory of $\Fun([n+1],\elTate(\D,\C))$ consisting of sequences of admissible monics
            \begin{equation*}
                L_0\hookrightarrow\cdots\hookrightarrow L_n\hookrightarrow V
            \end{equation*}
            where, for all $i$, $L_i\hookrightarrow V$ is the inclusion of a relative lattice (cf. Definition \ref{defi:relative_gr}).\footnote{To see that this is an exact category, observe that because $\Pro(\D)$ and $\Ind(\C)$ are closed under extensions in $\elTate(\D,\C)$, $\Gr_n^\le(\D,\C)$ is closed under extensions in $\Fun([n+1],\elTate(\D,\C))$.}
        \item Define the \emph{relative Sato complex} $\Gr^{\leq}_{\bullet}(\D,\C)$ to be the simplicial diagram of exact categories with $n$-simplices $\Gr_n^\le(\D,\C)$, with face maps $d_i$ given by omitting the $i^{th}$ relative lattice, and with degeneracy maps $s_i$ given by repeating it.
    \end{enumerate}
\end{definition}

Lemma \ref{lemma:relGrinC} allows for the following definition.
\begin{definition}
    Let $\D$ be idempotent complete and let $\C\subset \D$ be a right s-filtering subcategory. The \emph{categorical relative index map} is the span of simplicial maps
    \begin{equation}\label{catrelindex}
        \elTate(\D,\C)\longleftarrow\Gr_\bullet^\le(\D,\C)\to^\Index S_\bullet(\C),
    \end{equation}
    where the left-facing arrow is given on $n$-simplices by the assignment
    \begin{equation*}
        (L_0\hookrightarrow\cdots\hookrightarrow L_n\hookrightarrow V)\mapsto V,
    \end{equation*}
    and $\Index$ is given on $n$-simplices by the assignment
    \begin{equation*}
        (L_0\hookrightarrow\cdots\hookrightarrow L_n\hookrightarrow V)\mapsto(L_1/L_0\hookrightarrow\cdots\hookrightarrow L_n/L_0).
    \end{equation*}
\end{definition}

We have an analogue of \cite[Prop. 3.3]{ArticleTate2} in this setting.
\begin{proposition}\label{prop:relgr}
    Let $\D$ be idempotent complete, and let $\C\subset\D$ be an extension-closed subcategory. Then the augmentation $\Gr_\bullet^\le(\D,\C)\to\elTate(\D,\C)$ of \eqref{catrelindex} induces an equivalence
    \begin{equation*}
        |\Gr_\bullet^\le(\D,\C)^\times|\to^\cong\elTate(\D,\C)^\times.
    \end{equation*}
\end{proposition}
\begin{proof}
    Proposition \ref{prop:rel_Gr-directed} implies that the relative Grassmannian is a directed partially ordered set. This implies that the geometric realisation of its nerve (also known as classifying space) is contractible. The simplicial groupoid $\Gr_\bullet^\le(\D,\C)^\times$ is equivalent to the nerve of the category $\Gr^{\le}(\D,\C)$, whose objects are pairs $(V,L)$, with $V \in \elTate^{\grp}(\D,\C)$, and $L \in \Gr(V)$. Morphisms are given by commutative diagrams
\[
\xymatrix{
L \ar@{^(->}[r] \ar@{^(->}[d] & M \ar@{^(->}[d] \\
V \ar[r]^{\simeq} & W
}
\]
By virtue of Quillen's Theorem A we obtain that the functor $$\Gr^{\le}(\D,\C)^{\grp} \rightarrow \elTate(\D,\C)^{\grp}$$ induces an equivalence of classifying spaces.
\end{proof}

Along with Lemma \ref{lemma:SrelInd} this implies an analogue of Corollary \cite[Cor. 3.6]{ArticleTate2}.
\begin{corollary}
    Let $\D$ be idempotent complete, and let $\C\subset\D$ be a right s-filtering subcategory. The categorical relative index map determines a map of infinite loop spaces
    \begin{equation}\label{relinfiniteloop}
        B\Index\colon |S_\bullet\elTate(\D,\C)^\times|\to|S_\bullet S_\bullet(\C)^\times|
    \end{equation}
    which fits into a homotopy commuting square
    \begin{equation}\label{relconnKtriang}
        \xymatrix{
            \elTate(\D,\C)^\times \ar[d] \ar[rr]^{\Index} && |S_\bullet(\C)^\times| \ar[d]^\simeq \\
            \Omega|S_\bullet\elTate(\D,\C)^\times| \ar[rr]_{\qquad \Omega B\Index} && \Omega|S_\bullet S_\bullet(\C)^\times| }.
    \end{equation}
    We refer to the looping of the bottom horizontal map as the \emph{$K$-theoretic relative index map}.
\end{corollary}

This brings us to the main property of the relative Index map.

\begin{theorem}\label{thm:relindexloc}
    Let $\D$ be idempotent complete, and let $\C\subset\D$ be a right s-filtering subcategory. Then the $K$-theoretic relative index map fits into a homotopy commuting diagram
    \begin{equation*}
        \xymatrix{
            \Omega K_{\elTate(\D,\C)}  \ar[rr]^{\Omega^2 B\Index} && K_{\C}\\
            \Omega K_{\D/\C} \ar[u] \ar[urr]_{\partial}                     }.
    \end{equation*}
\end{theorem}
\begin{proof}
    The categorical relative index map fits into a 2-commuting diagram
    \begin{equation*}
        \xymatrix{
            \elTate(\D,\C) \ar[d] && \Gr_\bullet^\le(\D,\C) \ar[ll] \ar[d] \ar[drr]\\
            \Pro(\D)/\C  && S^r_\bullet(\C\subset\Pro(\D)) \ar[ll] \ar[rr] && S_\bullet(\C) \\
            \D/\C \ar[u] && S^r_\bullet(\C\subset\D), \ar[u] \ar[ll] \ar[urr]
        }
    \end{equation*}
    and the map $\Gr_\bullet^\le(\D,\C)\to S^r_\bullet(\C\subset\Pro(\D))$ is given on $n$-simplices by the assignment
    \begin{align*}
        (L_0\hookrightarrow\cdots\hookrightarrow L_n\hookrightarrow V)\mapsto (L_1/L_0\hookrightarrow \cdots L_n/L_0;L_0\hookrightarrow\cdots\hookrightarrow L_n).
    \end{align*}
    Similarly, using Lemma \ref{lemma:SrelInd}, we see that there exists a 2-commuting diagram
    \begin{equation*}
        \xymatrix{
            S_\bullet(\elTate(\D,\C))^\times \ar[d] && \Gr_\bullet^\le(S_\bullet\D,S_\bullet(\C))^\times \ar[ll] \ar[d] \ar[drr]\\
            S_\bullet(\Pro(\D)/\C)^\times && S_\bullet S^r_\bullet(\C\subset\Pro(\D))^\times \ar[ll] \ar[rr] && S_\bullet S_\bullet(\C)^\times \\
            S_\bullet(\D/\C)^\times \ar[u] && S_\bullet S^r_\bullet(\C\subset\D)^\times \ar[ll] \ar[u] \ar[urr]
        }
    \end{equation*}
    Geometrically realizing and taking the double loop spaces, we obtain a homotopy commuting diagram
    \begin{equation*}
        \xymatrix{
            \Omega K_{\elTate(\D,\C)} \ar[d] && \Omega^2|\Gr_\bullet^\le(S_\bullet\D,S_\bullet(\C))^\times| \ar[ll]^\simeq \ar[d] \ar[drr]\\
            \Omega K_{\Pro(\D)/\C} && \Omega^2|S_\bullet S^r_\bullet(\C\subset\Pro(\D))^\times| \ar[ll]^\simeq \ar[rr] && K_{\C} \\
            \Omega K_{\D/\C} \ar[u] && \Omega^2|S_\bullet S^r_\bullet(\C\subset\D)^\times| \ar[u] \ar[ll]^\simeq. \ar[urr]
        }
    \end{equation*}
    Note that the lower two left-facing maps are equivalences by Schlichting's Proposition \ref{prop:Schlichting}. The first left-facing map is an equivalence by virtue of Waldhausen's $K_{\C} \cong \Omega|S_{\bullet}\C^{\grp}|$. To be more precise, we apply first Proposition \ref{prop:relgr} to deduce that $|\Gr_{\bullet}^{\le}(S_{\bullet}\D,S_{\bullet}\C)| \cong \elTate(S_{\bullet}\D,S_{\bullet}\C)^{\grp}$, and then Lemma \ref{lemma:SrelInd} to deduce $$\elTate(S_{\bullet}\D,S_{\bullet}\C)^{\grp} \cong S_{\bullet}\elTate(\D,\C)^{\times}.$$
    After inverting the left-facing equivalences, we obtain a homotopy commuting diagram
    \begin{equation*}
        \xymatrix{
            \Omega K_{\elTate(\D,\C)} \ar[d] \ar[drr]\\
            \Omega K_{\Pro(\D)/\C} \ar[rr] && K_{\C} \\
            \Omega K_{\D/\C}. \ar[u] \ar[urr]
        }
    \end{equation*}
    By Theorem \ref{thm:boundary}, it suffices to prove that the map
    \begin{equation*}
        \Omega K_{\elTate(\D,\C)}\to\Omega K_{\Pro(\D)/\C}
    \end{equation*}
    is an equivalence. We derive this from Proposition \ref{prop:relTatequot} as follows. To wit, consider the 2-commuting diagram of exact categories
         \begin{equation*}
                \xymatrix@=9pt{
                   \C \ar[rr] \ar[dr] \ar[dd] && 0 \ar[dr]\ar'[d][dd]\\
                    & \Pro(\D) \ar[rr] \ar[dd] && \Pro(\D)/\C \ar[dd]^\simeq \\
                    \Ind(\C) \ar'[r][rr] \ar[dr]  && 0 \ar[dr]\\
                    & \elTate(\D,\C) \ar[rr] && \elTate(\D,\C)/\Ind(\C).
                }
            \end{equation*}
    where the equivalence is that of Propositions \ref{prop:relTatequot}. Applying $K$-theory, we obtain a commuting diagram in the stable $\infty$-category of spectra
         \begin{equation*}
                \xymatrix@=9pt{
                   \Kk_{\C} \ar[rr] \ar[dr] \ar[dd] && 0 \ar[dr]\ar'[d][dd]\\
                    & \Kk_{\Pro(\D)} \ar[rr] \ar[dd] && \Kk_{\Pro(\D)/\C} \ar[dd]^\simeq \\
                    \Kk_{\Ind(\C)} \ar'[r][rr] \ar[dr]  && 0 \ar[dr]\\
                    & \Kk_{\elTate(\D,\C)} \ar[rr] && \Kk_{\elTate(\D,\C)/\Ind(\C)}.
                }
            \end{equation*}
Note that, of the entries in the diagram, only $\Kk_{\C}$ has non-vanishing $\pi_0$. Along with Theorem \ref{thm:schlichting}, this shows that the top face is bi-cartesian. We claim that the commuting cube above is bi-cartesian as well. By virtue of \cite[Lemma 1.2.4.15]{Lurie:HA}, this is equivalent to the induced square of cofibres of the vertical morphisms being bi-cartesian. However, Propositions \ref{prop:relTatequot} and \ref{prop:z2} allow one to compute those cofibres as the diagram obtained by applying $\Kk_-$ to the square of exact categories
\[
\xymatrix{
\Ind(\C)/\C \ar[r]^-{\cong} \ar[d] & \elTate(\D,\C)/\Pro(\D) \ar[d] \\
0 \ar[r] & 0,
}.
\]
Both this square and the resulting square of $K$-theory spaces are bi-cartesian. Since the top face and the commuting cube itself are bi-cartesian, we see that the bottom face has to be bi-cartesian as well.

The Eilenberg swindle implies that $\Kk_{\Ind(\C)}\simeq 0$, and we conclude that $\Kk_{\elTate(\D,\C)}\simeq \Kk_{\elTate(\D,\C)/\Ind(\C)}\simeq \Kk_{\Pro(\D)/\C}$ as claimed.
\end{proof}

\begin{corollary}\label{cor:G-theory}
    Let $X$ be a Noetherian scheme, and let $W\subset Z\subset X$ be a flag of closed subschemes. Then the diagram of exact categories
    \begin{equation*}
        \xymatrix{
            \Coh(W) \ar[rr] \ar[d] && \Coh(Z) \ar[rr] \ar[d] && \Coh(Z\setminus W) \ar[d]^{\mathbf{T}_Z} \\
            \Coh_W(X)\ar[rr] && \Coh_Z(X) \ar[rr] && \elTate(\Coh_Z(X),\Coh_W(X))
        }
    \end{equation*}
    determines a homotopy commuting triangle
    \begin{equation*}
        \xymatrix{
            \Omega K_{\Coh(Z\setminus W)} \ar[d]_{\mathbf{T}_Z} \ar[rr]^\partial && K_{\Coh(W)}\\
            \Omega K_{\elTate(\Coh_Z(X),\Coh_W(X))} \ar[urr]_{\Omega^2 B\Index}
        }.
    \end{equation*}
    Here we define the vertical map $K_{\Coh(Z\setminus W)} \rightarrow K_{\elTate(\Coh_Z(X),\Coh_W(X))}$ to be the composition of the map given by the exact functor $\Coh(Z\setminus W) \rightarrow \elTate(\Coh_Z(X),\Coh_W(X))/\Ind(\Coh_W(X))$ of Corollary \ref{cor:T_Z}, and the inverse of the homotopy equivalence $$K_{\elTate(\Coh_Z(X),\Coh_W(X))} \rightarrow K_{\elTate(\Coh_Z(X),\Coh_W(X))/\Ind(\Coh_W(X))}.$$
\end{corollary}
\begin{proof}
    We have a nested chain of Serre subcategories $\Coh_W(X)\subset\Coh_Z(X)\subset\Coh(X)$. By \cite[Example 1.7]{MR2079996}, Serre subcategories are right s-filtering. Therefore, we can take $\D=\Coh_Z(X)$ and $\C=\Coh_W(X)$ and apply Theorem \ref{thm:relindexloc} to obtain a commuting triangle in the $\infty$-category of spaces
    \begin{equation*}
        \xymatrix{
            \Omega K_{\elTate(\Coh_Z(X),\Coh_W(X))} \ar[rr]^-{\Omega^2 B\Index} && K_{\Coh_W(X)}\\
            \Omega K_{\Coh_Z(X)/\Coh_W(X)} \ar[u] \ar[urr]_\partial
        }.
    \end{equation*}
    By devissage \cite[Section 5]{MR0338129}, we have equivalences $K_{\Coh(W)}\simeq K_{\Coh_W(Z)}\simeq K_{\Coh_W(X)}$ and $K_{\Coh(Z)}\simeq K_{\Coh_Z(X)}$. We also have an equivalence $\Coh(Z)/\Coh_W(Z)\simeq\Coh(Z\setminus W)$ \cite[Chapter 5]{MR0232821}. Together with Quillen's localization sequence (the precursor of Schlichting's localization theorem for abelian categories), these equivalences determine an equivalence $\Omega K_{\Coh_Z(X)/\Coh_W(X)}\simeq\Omega K_{\Coh(Z\setminus W)}$, which is compatible with the boundary maps of the localization sequences. Applying these equivalences to the commuting triangle above, we obtain the commuting triangle of the desired type.
\end{proof}

\begin{proof}[Proof of Theorem \ref{thm:main_intro}]
Recall the exact functors
$$\mathbf{C}_Z\colon \Coh(X) \rightarrow \Pro(\Coh_Z(X))$$
and
$$\mathbf{T}_Z\colon \Coh(X \setminus Z) \rightarrow \elTate(\Coh(X),\Coh_Z(X))/\Ind(\Coh_Z(X))$$
from Proposition \ref{prop:C_Z} and Corollary \ref{cor:T_Z}. It follows from the construction of Corollary \ref{cor:T_Z} that for every coherent sheaf $\F$ on $X$, the inclusion
$$\mathbf{C}_Z(\F) \hookrightarrow \mathbf{T}_Z(\F|_{X \setminus Z})$$
defines a relative lattice. Therefore, we obtain a canonical map
\[
\Gr^{\le}_{\bullet,\bullet}(X,Z) \rightarrow \Gr_{\bullet}^{\le}(S_{\bullet}\Coh(X),S_{\bullet}\Coh_Z(X)),
\]
and the assertion of the theorem follows from Corollary \ref{cor:G-theory}, and the fact that the poset of coherent subsheaves of $j_*\F|_{X \setminus Z}$, which extends $\F|_{X \setminus Z}$ is filtered (indeed, if $\mathcal{G}_1,\mathcal{G}_2 \subset j_*\F|_{X \setminus Z}$ are coherent subsheaves, then so is $\mathcal{G}_1 + \mathcal{G}_2$).
\end{proof}

\begin{acknowledgement}
We would like to thank Theo B\"uhler for very helpful correspondence.
\end{acknowledgement}

\bibliographystyle{amsalpha}
\bibliography{master,ollinewbib}
\end{document}